\pdfoutput=1
\documentclass[11pt,reqno]{amsart}
\usepackage[paperheight=279mm,paperwidth=18cm,textheight=26cm,textwidth=14cm,includehead]{geometry}
\usepackage[T1]{fontenc}
\usepackage[utf8]{inputenc}
\usepackage[unicode]{hyperref}
\hypersetup{
bookmarks=true,
colorlinks=true,
citecolor=[rgb]{0,0,0.5},
linkcolor=[rgb]{0,0,0.5},
urlcolor=[rgb]{0,0,0.75},
pdfpagemode=UseNone,
pdfstartview=FitH,
pdfdisplaydoctitle=true,
pdftitle={Bi-parameter Carleson embeddings with product weights},
pdfauthor={Arcozzi, Mozolyako, Psaromiligkos, Volberg, Zorin-Kranich},
pdflang=en-US
}

\usepackage[style=alphabetic,maxalphanames=4]{biblatex}
\usepackage{amssymb}
\usepackage{amsmath}
\usepackage{amsthm}
\usepackage{amsfonts}
\usepackage{mathtools}

\usepackage{tikz}
\usetikzlibrary{patterns}

\DeclarePairedDelimiterX\Set[1]\{\}{%

#1
}

\DeclarePairedDelimiterX\innerp[2]{\langle}{\rangle}{#1,#2}

\numberwithin{equation}{section}
\theoremstyle{plain}
\newtheorem{theorem}[equation]{Theorem}

\newtheorem{lemma}[equation]{Lemma}

\newtheorem{conj}[equation]{Conjecture}

\theoremstyle{remark}
\newtheorem{remark}[equation]{Remark}

\theoremstyle{definition}
\newtheorem{definition}[equation]{Definition}

\DeclareMathOperator{\sign}{sign}

\def\sign{\operatorname{sgn}}

\def\Ker{\operatorname{Ker}}

\def\dist{\operatorname{dist}}

\newcommand{\al}{\alpha}
\newcommand{\eps}{\varepsilon}
\newcommand{\la}{\lambda}
\newcommand{\vf}{\varphi}

\newcommand{\cL}{\mathcal L}

\newcommand{\cP}{\mathcal P}

\newcommand{\bP}{\mathbb P}
\newcommand{\bR}{\mathbb R}

\newcommand{\bZ}{\mathbb Z}

\newcommand{\bS}{\mathbb S}

\begin{document}

\title[Orthogonality in normed spaces]%
{Orthogonality in normed spaces}
\author[B.~Burshteyn]{Boris Burshteyn}
\address[B.~Burshteyn]{}
\email{boris997@astound.net}
\author[A.~Volberg]{Alexander Volberg}
\thanks{AV is partially supported by the NSF grant  DMS 1900268}
\address[A.~Volberg]{Department of Mathematics, MSU, East Lansing, MI. 48823 and Hausdorff Center of Universit\"at Bonn}
\email{volberg@math.msu.edu}
%
%
\begin{abstract}
Motivated by the questions  in the theory of Fredholm stability in Banach space and Kato's strictly singular operators we answer several natural  questions concerning  ``orthogonality'' in normed spaces and the properties of metric projections. What the reader will see below might have benn known long ago, but we did not find it in the literature. Some open (for us)  questions are formulated at the end of Sections 7 and 9.
\end{abstract}
\maketitle

\section{Introduction}

The geometry of finite-dimensional normed spaces might be viewed upon from two perspectives: finite-dimensional subspaces of infinite-dimensional spaces, and finite-dimensional spaces proper, without the enveloping infinite-dimensional space. The former category includes divers areas such as bases in Banach spaces, extensions of linear operators, complemented subspaces and inherently-indecomposable spaces, and many more. The latter explores a ``pure'' geometry of normed spaces intrinsically connected with the properties of finite-dimensional convex symmetric bodies which define the norm of the space. 
A relatively recent comprehensive overview of the geometry of Banach spaces is found in the compendium by W.B. Johnson and J. Lindenstrauss \cite{JL}. 

\medskip

This present paper explores one aspect of these directions which, to our best knowledge, has not attracted much attention so far.  It is related to  a result of Krein--Krasnoselski--Milman from \cite{KKM}: 

\begin{theorem}\textup(Theorem 2 of \cite{KKM}\textup)
\label{kkm}
Given any two finite-dimensional subspaces $F$, $E$ of a normed space, if $\dim F > \dim E$, then $F$ has a unit vector whose distance to $E$ is  $1$. 
\end{theorem}

This interesting theorem has been used to prove many facts in Functional Analysis (say, in the theory of Fredholm stability and in construction of bases, just to mention a few). Its proof is based on Theorem  by Lusternik--Schnirelmann  \cite{LS}  proved independently  by Borsuk--Ulam \cite{Bo}.

\medskip

What we consider below is related to the subject of minimal projections (projections with norm $1$ in normed spaces). Many authors have studied their properties in the context of functional analysis and approximation theory \cite{Fr}, \cite{KT}, \cite{KST}, \cite{LO}.  It turned out that the question of minimal projections is related to the question of maximal projections. Paper \cite{TK} explains this.

As an  application of the main result of \cite{TK},  the author observes that every $n$-dimensional normed space $X$ which has an $n-1$-dimensional subspace $Y$ with the maximal possible relative projection constant   of $X$ onto $Y$ (=$2-\frac2n$ according to Bohnenblust, \cite{Bohn})  has also a $2$-dimensional subspace with minimal possible relative projection constant $1$ (see Corollary 2.6, \cite{TK}).

\medskip

The relation with what follows lies in the following fact: let  $F$,  $K$  be  subspaces of finite-dimensional normed space $X$, and let the projection of $E=\text{span} (F, K)$ onto $K$ parallel to $F$ has norm $1$. Then $K$ is orthogonal to $F$ (see Definition \ref{orthdef} below). The converse is also true, orthogonality of $K$ to $F$ means the minimality of the norm of the  projection of $E$ onto $K$ parallel to $F$. In particular,  from Theorem \ref{kkm} below it follows that for any $n$-dimensional normed space $X$ and for any  subspace  $Y$ of dimension $n-1$,  we can find a $1$-dimensional subspace $K$, such that the projection onto $K$ parallel to $Y$ has norm $1$.

However, we will show that if  the words ``for any  subspace  $Y$ of dimension $n-1$'' is replaced by the words ``for any  subspace  $Y$ of dimension $n-2$'', then finding ``the right'' subspace $K$ of  dimension $2$ having the norm of the projection onto it parallel to $Y$ with norm $1$  can be impossible.

\medskip

Two counter-examples in this paper show that \cite{KKM} theorem cannot be extended from finding one unit vector to finding a $k$-dimensional subspace in $F$ which is ``orthogonal'' to $E$ with $k > 1$. The first example is in the infinite-dimensional space $L^p$, the second one is in a finite-dimensional normed space $\ell^p_N$. The paper is organized as follows:
Section 2 formulates questions 1 and 2.
Section 3 gives a definition of orthogonality.
Section 4 presents a counter-example to question 1.
Section 5 presents  partial result based on Borsuk--Ulam or Lusternik--Schnirelmann theorem and explains its link to question 2.
Section 6 presents a counter-example to question 2 in the case of dimension 2 orthogonality.
Section 7 presents further counter-examples.
Section 8 explains the  connection between \cite{KKM}, question 2, and the  cohomology theory.
Section 9 explains the source of both questions from the point of view of the Fredholm stability theory and formulates still unresolved questions related to questions 1 and 2.

{\bf Acknowledgement.} We are grateful to Mark Rudelson for valuable discussions.

\section{Two questions}

\begin{definition}
\label{orthdef}
Let $X_1$, $X_2$ are two subspaces of a normed space $Y$. We say that $X_1$ is orthogonal to $X_2$ if for every unit vector in $X_1$ the distance to $X_2$ is $1$. We say that $X_1$ is $\eps$-almost orthogonal to $X_2$ if this distance is $\ge 1-\eps$ for every unit vector in $X_1$.
\end{definition}

\begin{remark}
Notice that if $X_1$ is orthogonal to $X_2$ then it might very well happen that $X_2$ is not orthogonal to $X_1$.
\end{remark}

Question 1. Let $E$ and $F$ be two finite dimensional  subspaces of a normed space $X$, and $\dim E<\dim F$. Let $\eps>0$ be given.   Is it true that
there is a non-trivial subspace $K$ of $F$ such that $E$  is $\eps$-orthogonal to $K$?

\medskip

Question 2. Let $E$ and $F$ be two finite dimensional normed subspaces of a normed $X$, and $\dim E<\dim F$. Let $\eps>0$ be given. Is it true that
there is a non-trivial subspace $K$ of $F$ such that $K$  is $\eps$-orthogonal to $E$? If so, when $\eps$ can be made $0$? Can we find such $K$ with the property that $\dim K= \dim F-\dim E$? If not what maximal dimension of $K$ is possible?

\medskip

\section{Preliminaries: metric projections and orthogonality in Banach spaces}
\label{metric}

Let $X$ be (it is convenient to think that it is finite dimensional, but this is not necessary) a normed space with smooth strictly convex unit ball. Then for every vector  $e\in X$, $\|e\|=1$,  there exists a special functional, call it $f_e$, that satisfies : 1) $f_e(e)= 1$, 2) $\|f_e\|=1$.

Geometrically we can think about $f_e$ as a vector $n_e$  orthogonal to the unit sphere of $X$ at point $e$, and normalized to have 
$$
\max_{e'\in S} (n_e, e')=1\,,
$$
where $(\cdot, \cdot)$ is the usual Hilbert duality.

Given $e$ and linear subspace $\cL$ (let $\dim \cL<\infty$), we  define the metric projection of $e$ onto $\cL$ as a vector $x\in \cL$ such that
$$
\dist (e, \cL) =\|e-x\|\,.
$$
In spaces with smooth strictly convex balls, this $x$ is unique, call it $\cP_\cL e$.

Suppose we are given $e, v\in X$,  $\|e\|=1\|,  \|v\|=1$, let $\cL_v$ be the one dimensional linear space spanned by $v$. Then the metric projection of $e$ onto $\cL_v$ can be written by formula:
\begin{equation}
\label{proj}
\cP_{\cL_v} e = f_e(v) v\,.
\end{equation}
In particular, (remember that $\|e\|=1$)
\begin{equation}
\label{ker}
v\in \Ker(f_e) \quad\text{iff}\quad \dist (e, \cL_v)=1\,.
\end{equation}

In particular, unit vector $e$ is orthogonal to $\Ker (f_e)$ (but not the vice versa).
\begin{remark}
It is important to not use formula $f_v(e) v$ for $\cP_{\cL_v} e$.
\end{remark}

\section{About question 1}
\label{q1}

Everything is real valued below.
Let $X$  be a smooth with strictly convex unit sphere Banach space, like $L^p$, $p\neq 1, \infty$. Actually we will construct counterexample in any $L^p(0,1)$, such that
$p\neq 1, \infty$ and $p\neq 2$. But below we consider only the case of $p=3$. 

Now $E_n$ will be a space of polynomials of degree $\le n-1$ in $L^p(0,1)$. Consider the unit sphere in $E_n$. Enumerate polynomials in the unit sphere of $E_n$ by some index set: $\{P_i\}_{i\in I}$.
For every polynomial $\{P_i\}_{i\in I}$ choose the functional of norm $1$, in other words, a function $f_i$ in $L^q(0,1)$, $q=\frac{p}{p-1}$ of norm $1$ such that $f_i(P_i):= \int_0^1 P_i(t) f_i(t) dt=1$. It is easy to write down such a unique $f_i$:
$$
f_i(t) := |P_i(t)|^{p-1} \sign P_i(t)\,.
$$
Here is a key observation: the dimension of the linear span $\cL\{f_i\}_{i\in I}$ is huge, actually it is in fact $>>n$ in all the cases except for the case $p=2$, where $f_i =P_i$ and dimension of the span is $n$.

\begin{lemma}
\label{codim}
Let $p=3$. Let $n$ be odd. Then $\dim\cL\{f_i\}_{i\in I}\ge n+\frac{n}2-\frac12$.
\end{lemma}

\begin{proof}
If $\dim\cL\{f_i\}_{i\in I}=\infty$ we are done. 
Otherwise we are in a finite dimensional situation and let $\{f_{i_j}\}, j=1, \dots, d,$ be a smallest linearly independent collection.
Let $n-1=2k$ with odd $k$ for example.
Let $J\subset (0, 1)$ be a segment without zeros of $P_{i_j}$. Choose $P\in E_n$ of norm $1$ and we can diminish $J$ not to have a zero of $P$ on it. Now we can write
$$
\sum_{j=1}^d A_j \cdot(P_{i_j})^2(t)\equiv P^2(t), \quad t\in J
$$
with some constants $A_j$ (the signs on $J$ are absorbed into $A_j$). Then the same  is true on $(0,1)$. 
Thus  the dimension of the linear span of all $P^2$, $\deg P\le n-1$, is at most $d$. But $1, x^2, x^4, \dots, x^{2(n-1)}$ and $(1+x)^2, (1+x^3)^2,...,
(1+ x^k)^2$ are all linearly independent. So $d\ge n+\frac{n}2-\frac12$.
\end{proof}

Now let us use that this dimension is  large. Consider hyperplanes in $X=L^3$, $H_i:= Ker f_i$. Consider
$$
H:= \cap_{i\in I} H_i,
$$
it is a linear subspace in $X=L^3(0,1)$. By the previous lemma
\begin{equation}
\label{codim}
\text{codim} H \ge d = n +\frac{n}2 -\frac12\,.
\end{equation}

Consider the linear subspace $Y_n$ generated by $E_n$ and $H$. Now we are ready to choose $F_n$.  Notice that by construction $H\cap E_n =\{0\}$. Let us choose a subspace  $F_n$ of $Y_n$, $\dim F_n =d$, that intersects $E_n$ only in $\{0\}$ and intersects $H$ only in $\{0\}$. This is of course possible by \eqref{codim}.  This is our $F_n$.

Suppose there is a non-trivial subspace $K$ of $F_n$ such that $E_n$ is orthogonal to $K$. Then for any unit vector $v$ in $K$ we have
\begin{equation}
\label{dist1}
inf_\la \dist ( e-\la v) = 1, \quad \forall e\in E_n, \|e\|=1\,.
\end{equation}

Fix $v\in  K$.
Then the functional on the span of $e, v$, given by $\tilde f_e(\al e+\beta  v) =\al$, has three properties: 1) on this span $\|\tilde f_e\|=1$, 2) $\tilde f_e(e)=1$, 3) $\tilde f_e(v)=0$. The first property follows immediately from \eqref{dist1}. Extend those functionals
to $f_e$ on $L^3(0,1)$ with the same norm.  Consider $e=e_i, i\in I$. Then the intersection of kernels of $f_{e_i}$ is $H$ and  by construction $H\cap F_n=\{0\}$. On the other hand,
$$
f_{e_i}(v) =0, \quad \forall i\in I\,.
$$
So $v\in H$, but by construction $v\in F_n$. Contradiction. So $K$ is trivial.

We have shown that for the constructed  pair $(E_n, F_n)$ there is no non-trivial $K$, which is subspace of $F_n$, such that $E_n$ is orthogonal to $K$.

\bigskip

Suppose that for every $\eps_m=1/m$ there is a unit vector $v_m\in F_n$ such that $E_n$ is $1/m$-almost orthogonal to $\cL_{v_m}$, the span of one vector $v_m$. Choose $v$ such that $\|v-v_{m_k}\|\to 0$. For any unit vector $e\in E_n$, there is $\la_k$ such that
$$
\|e-\la_k v_{m_k}\|\ge 1-\frac1{m_k}
$$
and $\la_k$ is such that $|\la_k|\le 2$. Without loss of generality we may think that $\{\la_k\}$ converges. Then 
$$
\inf_\la\|e -\la v\|=1\,.
$$
This means that $E_n$ is orthogonal to $\cL_{v}$, but we proved that this is not the case.

We conclude that there exists $\eps_n$ in our example such that there is no nontrivial subspace $K$ in $F_n$ such that $E_n$ is $\eps_n$-almost orthogonal to $K$.

\section{About question 2}
\label{q2}

It tuns out that the following result appeared in \cite{KKM} as Theorem 2 of this paper. The main ingredient is \cite{LS} and its Lusternik--Schnirelmann theorem also known as Borsuk--Ulam theorem.

\begin{theorem}
\label{borsuk}
Let $E, F$ two subspaces  of a finite dimensional normed space $Y$ with smooth norm, $\dim E < \dim F$. Then there exists unit vector in $F$ such that
the space $\cL(v)$ spanned by it is orthogonal to $E$, in other words $\dist(v, E)=1$.
\end{theorem}

\begin{proof}
As the norm of $Y$ is smooth, for any unit vector $v$, there exists unique vector $e\in E$ that minimizes $\|v-e\|$. Thus we get the mapping $e= \xi(v)$ from the unit sphere $S_F$ of $F$ into $E$. Obviously $\xi(-v)=-\xi(v)$. Also it is clear that $\xi: S_F\to E$ is continuous. Also $\dim S_F\ge \dim E$. By Borsuk--Ulam lemma, see \cite{R}, Chapter 12, there exists the pre-image of $0$: $\exists v\in S_F$ such that 
$\xi(v)=0$. This means that $v$ is orthogonal to $E$.
\end{proof}

Let now $E, F, Y$ are as before, but the norm in $Y$ is not smooth. Then approximate the unit ball $B_Y$ extremely well by the 
smooth strictly convex symmetric body $\tilde B_Y$. Construct $v$ from Theorem \ref{borsuk} in this new metric. Then obviously 
$v$ is $\eps$-orthogonal to $E$ in original metric with $\eps$ as small s we wish and depends only on how well we approximated $B_Y$ by $\tilde B_Y$.

\medskip

{\it The second proof of Theorem \ref{borsuk}}. In $F$ we wish to find  a unit vector $v$ such that it is orthogonal to $E$. This means that $\cP_E(v) ={0}$, which means that for any $e\in E, \|e\|=1$ we have $\cP_{\cL_e} v={0}$. By \eqref{proj} it is the same as to write
$$
f_v(e) e={0}\quad \forall e\in E, \|e\|=1,
$$
But this is the same as to write
\begin{equation}
\label{EKer-v}
E\subset \Ker f_v\,.
\end{equation}
Suppose now that for every $v\in F, \|v\|=1,$ we do {\it not} have \eqref{EKer-v}.  Then consider the set of real numbers $J_v:=\{ f_v(e), \|e\|=1,  e\in E\}$. This set is {\it not} consisting of one point $\{0\}$ because we do {\it not} have \eqref{EKer-v}. It is compact.

Hence we can find $m(v)=\max\{x: x\in J_v\}$.  It is given by some $e_v$, $m(v) = f_v(e_v)$. Notice that for every $v$, vector $e_v$ is unique. In fact, if we have $e_v^1, e_v^2$, then normalized $\frac{e_v^1+e_v^2}{2}$ will exceed the maximum.

We obtained a map $m: S_F\to S_E$, from the unit sphere of $F$ into the unit sphere of $E$. It is very easy to see that 
$$
m(-v) = -m(v),
$$
and that $m$ is continuous. Borsuk--Ulam lemma says that this is impossible.
Hence, there is a unit vector $v\in F$ orthogonal to the whole $E$.

\begin{remark}
One could try to use the metric projections of $e\in E$ onto $\cL_v$, $\cP_{\cL_v}$, but $\cP_{\cL_v} e = f_e(v) v$, and it depends very much on $e$ when $v$ is fixed, and uniqueness of the optimal $e$ is elusive.
\end{remark}

\section{About question 2 with dimension higher than $1$. A counterexample}
\label{dimTwo}


Consider unit sphere in $\ell^3_N$
$$
S:= \{x:\sum_{i=1}^N |x_i|^3 =1\},
$$
consider its part $S_+:= S\cap \{x_i> 0, i=1, \dots, N\}$.
Let $\cL$ be a collection of dimension two subspaces in $\bR^N$. We wish to show that only 
a small number of them has the following property. We will call $L\in \cL$ having the property listed below ``bad'' subspaces. All others $L\in \cL$ will be called ``good''.


\medskip

Badness means the following: consider the intersection of $L\cap S$, it is a curve on $S$, and consider all unit normals
to $S$ at points of $L\cap S$. Span a linear subspace by all those normals (by the way the normals themselves represent a one-dimensional smooth submanifold of the round sphere $\bS^{N-1}$). Call this span $N(S, L)$, and if $\dim N(S, L)<3$ (then it is clear that $\dim N(S, L)=2$), we call $L$ bad. Otherwise if
\begin{equation}
\label{good}
\dim N(S, L) \ge 3\,,
\end{equation}
we call  $L$ good.

\bigskip



Space $\cL$ is a real Grassmannian $G(2, N)$. Denote the set of  bad $L$'s by $\cL_b$. It is very easy to see that it is a closed set. 

By dimension below we understand the Hausdorff dimension.
In the following, it should be clear from the context when $\dim X$ means the dimension of a vector space $X$ or the Hausdorff dimension of a subset  of a Grassmannian.

In the next lemma,  the dimensions means  the Hausdorff dimension.

\subsection{A bit of geometry}
\label{abit}
\begin{lemma}
\label{bad}
For any bad $L$ there is a good one, arbitrarily close to it, moreover, $\dim \cL_b< \dim \cL$.
\end{lemma}
\begin{proof}
Fix a bad $L$ and assume WLOG that $L$ intersects $S_+$. Let $L$ be given by
$$
Cx=0
$$
with matrix $C=\{c_{ij}\}_{i=1,\dots, N-2; j=1, \dots, N}$, whose rank is $N-2$.

Then $N(S, L)$ contains all vectors $e_2(x):= (x_1^2,\dots, x_N^2)$ for   every $x, x_i>0$, such that $\sum_{i=1}^N x_i^3=1 $ and $Cx=0$.
In fact, if a surface is given by $f(x)=1$, then the normal to this surface (not normalized) is given by vector $(f_{x_1}(x), \dots, f_{x_N}(x))$.



Without loss of generality let the minor $\hat C=\{c_{ij}\}_{i=1,\dots, N-2; j=3, \dots, N}$ has non-zero determinant. 
Then we express $x_3,\dots, x_N$ via  linear form of $(x_1, x_2)$ with coefficients depending
on $\{c_{ij}\}$.

We solve
$$
\hat C \hat x = x_1 c_{*, 1} + x_2 c_{*, 2},
$$
where $c_{*, 1} , c_{*, 2}$ are two first columns of $C$, $\hat x:=(x_3, \dots, x_N)$. We get the solution
$$
x_3= \gamma_{31} x_1 +\gamma_{32} x_2,\dots,  x_N= \gamma_{N1} x_1 +\gamma_{N2} x_2,
$$
where $\gamma_{k\ell}$ are real analytic functions of $c_{ij}$.

After these plug-ins, the vector
$e_2(x)$
becomes 
$$
e_2(x_1, x_2) = (x_1^2, x_2^2, Q_3(x_1, x_2),\dots, Q_N(x_1, x_2))\,,
$$
where $Q_k(x_1, x_2)$ are homogeneous polynomials of degree $2$:
$$
Q_k(x_1, x_2) = (\gamma_{k1} x_1 +\gamma_{k2} x_2)^2\quad k=3,\dots, N\,.
$$

Can such a collection $\{e_2(x_1, x_2)\}$, where point $(x_1, x_2)$ runs over $\Gamma$, spans only $2$ dimensional space, where $\Gamma=\{(x_1, x_2): P(x_1, x_2)=1\}$, and where $P$ is a (homogeneous) polynomial of degree $3$?  Here
$$
P(x_1,x_2) = x_1^3+x_2^3+ \sum_{k=3}^n(\gamma_{k1} x_1 +\gamma_{k2} x_2)^3\,.
$$

If collection $\{e_2(x_1, x_2)\}_{(x_1, x_2) \in \Gamma}\}$ spans only $2$ dimensional space, then there exists matrix
$D=\{d_{ij}\}_{i=1,\dots, N-2; j=1, \dots, N}$, whose rank is $N-2$, such that
$$
De_2(x_1, x_2) =0, \quad \forall (x_1, x_2): P(x_1, x_2)=1\,.
$$
This is possible only if 
\begin{equation}
\label{equiv}
De_2(x_1, x_2) \equiv 0\quad  \forall (x_1, x_2).
\end{equation}

Explanation: take a row of $D$, $d=(d_1, d_2, \dots, d_N)$, let $Q(x_1, x_2) :=
d_1x_1^2+ d_2 x_2^2+ d_3Q_3(x_1, x_2)+\dots+ d_n Q_N(x_1, x_2)$.  From the displays above we get that as soon as $P(x_1, x_2)=1$ we have $Q(x_1, x_2)=0$, where $Q$ is a homogeneous polynomial of degree $2$. This is of course impossible, unless $Q\equiv 0$. In fact, genuine curve $P=1$ on the plane cannot sit inside one straight line or two straight lines, when $P$ is a homogeneous polynomial of degree $3$.
In fact, suppose that polynomial $P$ contains both $x_1$ and $x_2$ and $P(x_1, x_2)=1\Rightarrow x_2=kx_1$ with constant $k$. Then we get $p(k) x_1^3=1$, where $p$ is a polynomial. The set of solutions of such equation cannot be contained in neither one nor two lines passing through the origin. If $P$ suddenly is  just $c x_1^3$ or $c x_2^3$, then again it is obvious that the set of solutions of $P=1$ cannot be contained in neither one nor two lines passing through the origin.

So we start with \eqref{equiv}. Assume now that the invertible square $(N-2)\times (N-2)$ minor $\hat  D$ of $D$ is the last $N-2$ rows of $D$ (it is really WLOG, as the reader will see from what follows). Notice that in the left hand side of \eqref{equiv} there is no term $x_1\cdot x_2$. This is important for what will now follow.

After opening  the brackets in quadratic forms involved in \eqref{equiv}, and collecting the terms in front of $x_1\cdot x_2$ only, equation \eqref{equiv} gives us  the following:
$$
\sum_{k=3}^N d_{k\ell}\cdot ( \gamma_{k1}\cdot \gamma_{k2}) = 0\,.
$$
As $\hat D$ is invertible, we conclude that
\begin{equation}
\label{0}
\gamma_{k1}\cdot \gamma_{k2}=0\quad\forall k=3, \dots, N.
\end{equation}

So as soon as $2$-dimensional subspace $L$ (parametrized by $C$ above) has $D$ as above (i.e. the space spanned by normals to $N(S, L)$ is  $2$-dimensional)  we immediately have \eqref{0}.  This means that $\gamma_{k1}=0, k\in A$, $\gamma_{k2}=0, k\in B$, and $A\cup B=\{3,\dots, N\}$. These are real analytic conditions of elements of matrix $C$, so those $\gamma's$ are either identically zero, or happen to be zero on positive co-dimension of $G(2, N)$.

Let us explain this. Fix an $(N-2)\times N$ matrix $C$ with rank $N-2$. WLOG we can think that the last $N-2$ columns form an invertible matrix $\hat C$.
Consider a small neighborhood of $C$, call its elements $B$. All $B$'s have the same property as $C$. Call this small neighborhood $M(C)$. We recall that
$G(2, N)$ can be identified with $M_{(N-2)\times N}/ Inv_{(N-2)\times(N-2)}$, where $Inv$ stands for invertible matrices. Let $g$ be an element of $G(2, N)$ that corresponds to $C$ under the identification mentioned above. Let $\pi$ be a canonical projection from $M(C)$ to $N(g)$, where $N(g)$ is a neighborhood of $g$ in $G(2, N)$.  Let us define map $f: M(C)\to M_{(N-2)\times 2}$,
$$
f(B) = \text{first two columns of} \, (\hat B^{-1} B)\,.
$$

We have three spaces $M(C), M_{(N-2)\times 2}, N(g)$, and two maps: $f: M(C)\to M_{(N-2)\times 2}$, $\pi:M(C)\to N(g)$.   Maps are real analytic with their images being of dimension $2(N-2)$ by construction.  Notice that we can close the commutative diagram by the third map
$$
F: N(g) \to M_{(N-2)\times 2}\,.
$$
Let us explain this: for a point $g'\in N(g)$ we have its $\pi^{-1}(g')$ consisting of certain matrices of the  form $A^{-1}B$. But by construction of $f$ it is obvious that $f(B)= f(A^{-1}B)$.
So we can push the map $f$ down to the map $F: N(g)\to M_{(N-2)\times 2}$ in such a way that
$$
\pi\circ F=  f\,.
$$
Then clearly $F:N(g) \to M_{(N-2)\times 2}$ is a real analytic diffeomorphism. 

Notice that matrix consisting of two columns $(\gamma_{*, 1}(C), \gamma_{*,2}(C)) $ is exactly    $f(C)= F(g)$. In the space $ M_{(N-2)\times 2}$ consider the sub-variety $M_0$ of matrices such that $a_{k1}\cdot a_{k2}=0, k=1, \dots, N-2$.  Then 1) its dimension $M_0$ is $N-2< \dim M_{(N-2)\times 2}= 2(N-2)$, 2) 
$F(g) = f(C)= (\gamma_{*, 1}(C), \gamma_{*,2}(C))$ belongs to $M_0$ by \eqref{0}. 

Hence, $g\in F^{-1}(M_0)$, and we are on a proper sub-variety of $G(2, N)$. We proved that $\cL_b$ locally lies on proper sub-variety of $G(2, N)$.

Hence, we conclude that just moving matrix $C$ very slightly we get that the corresponding $D$ cannot exist.

Moreover,  we proved that
\begin{equation}
\label{dim}
\dim \cL_b < \dim \cL\,.
\end{equation}

\bigskip

We are left to understand why it is enough to consider the case when the invertible square $(N-2)\times (N-2)$ minor $\hat  D$ of $D$ is the last $N-2$ rows of $D$.
Notice that we can think that two first columns $c_{*,1}, c_{*,2}$ of matrix $C$ are linearly independent. We fixed $C$ having last $N-2$ columns forming  full rank $N-2$.  If only matrices $C$ with linearly dependent  first and second columns  give the element of $\cL_b$, then $\dim \cL_b<\dim \cL$ automatically.

So we consider the case now when first two columns $c_{*,1}, c_{*,2}$ of matrix $C$ are linearly independent. Then columns $\gamma_{*,1}, \gamma_{*,2}$  are linearly independent, indeed $\gamma_{*, i} =\hat C^{-1} c_{*, i}$, $i=1,2$.

\medskip

Coming back to \eqref{equiv}: $De_2(x_1, x_2) \equiv 0\,\, \forall (x_1, x_2)$, we can now make a linear change of variable:  $\gamma_{*, i} \to \tilde\gamma_{*, i}$, $i=1,2$, in such a way that
$$
\tilde\gamma_{N-1, 1} =1, \tilde\gamma_{N, 1} =0; \quad \tilde\gamma_{N-1, 2} =0, \tilde\gamma_{N, 2} =1\,.
$$
Then we can rewrite  $De_2(x_1, x_2) \equiv 0\,\, \forall (x_1, x_2)$ as equation with forms $\tilde Q_k:=(\tilde\gamma_{k1} x_1+\tilde\gamma_{k2} x_2)^2$,$k=1,\dots, N-2$, $\vec Q:= \{\tilde Q_k\}_{k=1,\dots, N-2}$, $\vec d_{N-1} := \{d_{k, N-1}\}_{k=1,\dots, N-2}$,  $\vec d_{N} := \{d_{k, N}\}_{k=1,\dots, N-2}$ as follows
\begin{eqnarray}
\label{22}
&\hat D \vec Q =  \vec d_{N-1} (\tilde\gamma_{N-1, 1} x_1 + \tilde\gamma_{N-1, 2}x_2)^2 +  \vec d_{N} (\tilde\gamma_{N, 1} x_1 + \tilde\gamma_{N, 2}x_2)^2=\notag
\\
&= x_1^2 \vec d_{N-1} + x_2^2 \vec d_{N}
\end{eqnarray}
The right hand side does not have a mixed term $x_1\cdot x_2$, and the left hand side has invertible $\hat D$. Hence again we got a relationship akin to \eqref{0}, namely:
\begin{equation}
\label{22concl}
\tilde\gamma_{k1}\cdot \tilde\gamma_{k2}=0,\quad\forall k=1, \dots, N-2.
\end{equation}
As before we conclude \eqref{dim}.
\end{proof}

\bigskip

\begin{remark}
Notice that actually \eqref{equiv} can very well happen. Consider as $S$ a ``cubic-ellipsoid'' in $\bR^3$, $|x_1|^3 + |x_2|^3 + |x_3|^3=1$. Then $(x_1, x_2)$, $(x_1, x_3)$ and $(x_2, x_3)$ sections of the ellipsoid give  us the $2$-dimensional sections $L$ such that normals to $S\cap L$ are two dimensional.
These are the only three sections with this property on this cubic ellipsoid in $3D$.

On the usual quadratic ellipsoid in any dimension all $2$-dimensional sections $L$  will be such that normals to $S\cap L$ are two dimensional. This is because in the latter case vectors $e_2(x), x\in S\cap L,$ consist of linear rather than quadratic forms.
\end{remark}

\bigskip

\subsection{Constructing counterexample by using Lemma \ref{bad}}
\label{cex}

On $\cL_b\subset G(2, N)$ we defined the map
$$
 \vf: \cL_b\to G(2, N) \quad \vf(L)= N(S, L),
$$
which is clearly continuous (if $\xi_n\in S, \,\xi_n\to \xi$, then unit normal vectors to $S$ at $\xi_n$ converge to unit normal vectors to $S$ at $\xi$ by the smoothness of our $S$).

\medskip

On $G(2, N)$ one has a natural metric $\rho$. Let us show that $\vf$ is Lipschitz on $\cL_b$.

 In fact, let $L_1, L_2\in \cL_b$ be very close, say, $\rho(L_1, L_2) = \delta>0$,  
 and $N_1, N_2$ be their images under $\vf$. Consider a point $\xi\in L_1\cap S$, 
 and $v$ is unit normal at $\xi$ to $S$, then all points $\eta\in L_2\cap S$, 
 that lie at distance $C\delta$ have unit normals $u$  such that $\sin(u, v) \le C' \delta$. This is just by the smoothness of $S$.
 
 \medskip

 All vectors $u, v$ lie in their corresponding $2$-dimensional planes $N_2, N_1$. Hence, $\rho(N_1, N_2) \le C'\delta$.
 From Lipschitz property and \eqref{dim} we conclude that
 $$
 \dim \vf(G(2, N))\le \dim \cL_b< \dim G(2, N)\,.
 $$
 Thus
 \begin{equation}
 \label{in}
 \vf(G(2, N))\,\, \text{ is a compact strictly inside}\,\, G(2, N).
 \end{equation}
 
 \bigskip
 
 It is time to construct $E$, $F$. Put $F=\ell^3_N$. To choose $E$ we use \eqref{in}. Choose a point 
 $g$ in $G(2, N)$ such that 
 \begin{equation}
 \label{ch-g}
 g\notin \vf(G(2, N))\,.
 \end{equation}
 This $g$ is a two dimensional space of vectors $n$, and we consider the space $L_g$ of  those $e\in \bR^N$ that
 \begin{equation}
 \label{Lg}
 (e, n)=0, \quad \forall n\in g\Rightarrow  L_g =\cap_{n\in g} \Ker (n,\cdot)\,.
 \end{equation}
 here the duality $(\cdot,\cdot)$ is the usual Hilbert duality (which is also a duality between $\ell^3$ and $\ell^{3/2}$). Those $e$ constitute the subspace $L_g$  of $\ell^3_N$ of co-dimension $2$ we are looking for.
 
 In fact, take any $L$, a two dimensional subspace of $F$. If $L\in \cL_b$, then the support functionals $N(S, L)\neq g$, as $N(S, L)\in \vf(G(2, N))$ and $g\notin \vf(G(2, N))$. This means that functionals of $N(S, L)$ cannot be all annihilating $E$, in other words, we have
 \begin{equation}
 \label{notin1}
 E\not\subset \cap_{u\in L} \Ker f_u= \cap_{n\in N(S, L)} \Ker (n,\cdot)\,.
 \end{equation}
 In fact, \eqref{notin1}immediately follows fro \eqref{Lg}  and the fact that $g\ne N(S,L)$, the latter following from \eqref{ch-g}.
 
 The second case is that $L$,  a two dimensional subspace of $F$, is not in $\cL_b$, that is it is a good subspace. Then 
 $$ 
 \cap_{u\in L} \Ker f_u= \cap_{n\in N(S, L)} \Ker (n, \cdot)
 $$
 is a subspace of co-dimension at least $3$ as $N(S, L)$  spans  at least a $3$-dimensional subspace. But if co-dimension  of $ \cap_{u\in L} \Ker f_u$  is at least $3$ and co-dimension of $E$ in $F$ is exactly $2$, we have again \eqref{notin1}.
 
 \medskip
 
 Thus we constructed a smooth strictly convex space $F$ of dimension $N$, its subspace  $E$ of dimension  $N-2$, such that {\it there is no subspace of dimension} $2$ in $F$ that is orthogonal to $E$. But we know that the subspaces of dimension $1$ orthogonal to $E$ must exist, see Section \ref{q2}.
 
 \bigskip
 
 \begin{remark}
There is a possibility that one can avoid proving \eqref{dim}. It is quite easy to prove that $\cL_b\neq \cL$ for our ``cubic ellipsoid''. We just take $L$ given by
$C$ of a special form $c_{*,1}=(1, \dots, 1), c_{*,2}=(1, \dots, 1)$ and having $\hat C=Id$. It is very easy to prove that for such $L$, $N(S, L)$ generates a space of dimension at least $3$. Thus the map $\vf$ maps $G(2, N)$ to its proper compact subset (and $\vf$ is obviously continuous). But now a bit of topology, for which we are grateful to Michael Shapiro. Real manifold  $G(2, N)$ is a compact manifold that is also a  $CW$ complex. Compact manifold without boundary that is also a $CW$ complex cannot be homeomorphic to its proper part. This fact can be proved using senior homology groups.

If the image $\vf(\cL_b)$ is not the whole $G(2, N)$ we reason exactly as above, and construct $E, L\subset F$, such that $L, \dim L=2,$ is not orthogonal to $E$.

Thus, assume that the image of $\vf$ is the whole $G(2, N)$. It is a continuous image, and it is an easy  (but pretty) exercise to prove that $\vf$ does not glue the points of $\cL_b$. 
Hence $\vf^{-1}$ is a homeomorphism from  $G(2, N)$ to  $\cL_b$. Then $G(2, N)$ is homeomorphic to its proper part, which is a contradiction.
\end{remark}

\section{About subspace $E$ of $F$ of co-dimension $m$ that has no orthogonal space of dimension $k$, $m\ge k$}
\label{32}

Notice that in the previous Section we proved that dimension of bad subspaces $\cL_b$ of co-dimension $2$ in $F$ is at most $N-2$. This compact sits inside $\cL= G(2, N)$ of dimension $2(N-2)$. The dimension of all subspaces of dimension $3$ having inside one of the sub-space from $\cL_b$ is at most
$N-3+ N-2= 2N-5$. Call such objects $\cL_{bb}$. It sits in $G(3, N)$ whose dimension is $3(N-3)$, so if $N>4$ (and we are interested in large $N$) we will always be able to choose $g\in G(3, N)\setminus \cL_{bb}$.

Consider $E_g$ that corresponds to $g$, that is 
$$
E_g=\{e: (n,e)=0\, \forall n\in g\}\,.
$$
  Space $E_g$ a linear subspace of co-dimension exactly $3$ in $F$, and  if $L\in \cL_b$, and it happens that
\begin{equation}
\label{extra-orth}
E_g\subset \cap_{u\in L} \Ker f_u,
\end{equation}
we notice that by construction of $g$, $L\not\subset g$. But then $\dim\text{span}(L, g) \ge 4$.  Thus, by two last display formulas the co-dimension of $E_g$ is $4$, which is a contradiction. Hence \eqref{extra-orth} cannot hold.
So $2$ dimensional subspaces generated by bad $L$ cannot be orthogonal to $E$.

\medskip

Let now $L$ be good. If we would be able to prove  that for $L\in \cL\setminus \cL_b$ (that is for good $L$), we have that
\begin{equation}
\label{4} 
N- \dim\big(\cap_{u\in L} \Ker f_u\big)  = \dim\text{span}\{ f_u: u\in L\} \ge 4,
\end{equation}
then the same relationship 
$$
E_g\not\subset \cap_{u\in L} \Ker f_u,
$$
would obviously hold for good $L$. Then $E_g$ of co-dimension $3$ would be without orthogonal subspaces of dimension $2$. 

\medskip

But how to prove \eqref{4}?
Of course we have \eqref{good} by definition of goodness. But \eqref{4} is a stronger property.
One should imitate \eqref{22}, \eqref{22concl}. But it is not clear how to do this imitation.

\bigskip

In the previous part of this Section $m=3, k=2$. And we do not know whether we can always find $E, F$, $F=\ell^3_N$, $\dim F-\dim E =3$, such that no subspace of $F$ of dimension $2$ is orthogonal to $E$.  But below we will build the required example if $F=\ell^5_N$.

\subsection{Obstacle  of time $(3,3)$}
\label{33s}

It is easy to imitate \eqref{22}, \eqref{22concl}  and to build an example $E, F$, $F=\ell^3_N$,  $\dim F-\dim E =3$, and no subspace of $F$ of dimension $3$ is orthogonal to $E$.  In fact, we call bad and denote $\cL_{b3}$ the elements of $G(3, N)$ that have property $\dim N(S, L) <4$. And we call good those elements of $G(3, N)$ for which \eqref{4} holds.

Then we can repeat the previous Section almost verbatim because in this case we can just repeat the considerations that brought us \eqref{22concl}. In fact we would come to the analog of \eqref{22} but with matrix $\hat D$ being $(N-3)\times (N-3)$ matrix and with 3 terms in the right hand side of \eqref{22}: $\vec d_{N-2} x_1^2+ \vec d_{N-1} x_2^2 + \vec d_N x_3^2$. Again we have no mixed term in this expression, and  the analog of \eqref{22concl}  would follow.
In fact we have proved the following theorem.
\begin{theorem}
\label{mm}
For any $1<m$, we can choose large enough $N$ and $F=\ell^3_N$ such that there exists a subspace $E\subset F$, $\dim F-\dim E = m$ such that $F$ does not have a subspace of dimension $m$ orthogonal to $E$.
\end{theorem}

\medskip


\subsection{Counterexample for $m=3, k=2$ in $\ell^5_N$}
\label{32cex}

Now we work in $F=\ell^5_N$, and we wish to find a subspace $E$ of co-dimension $3$ such that
there is no $2$-dimensional space orthogonal to $E$.

For that we again consider all subspaces $S$ of $F$ of co-dimension $2$ and we  call good those for which
\begin{equation}
\label{4co}
\dim N(S, L) \ge 4\,.
\end{equation}
Otherwise $S$ is bad and the set of bad subspaces is again called $\cL_b$. This is a compact in $G(2, N)$ again.

\begin{lemma}
\label{ell54}
If $F=\ell^5_N$ then $\dim \cL_b <\dim G(2, N)$.
\end{lemma}

\begin{proof}

To prove lemma we need to repeat a piece of Subsection \ref{abit}.
We consider vectors $e_2(x)$ as before, but now they deserve the name $e_4(x)$. They
become 
$$
e_4(x_1, x_2) = (x_1^4, x_2^4, Q_3(x_1, x_2),\dots, Q_N(x_1, x_2))\,,
$$
where $Q_k(x_1, x_2)$ are homogeneous polynomials of degree $4$.  
$$
Q_k(x_1, x_2) = (\gamma_{k1} x_1 +\gamma_{k2} x_2)^4\quad k=3,\dots, N\,.
$$

Can such a collection $\{e_4(x_1, x_2)\}_{(x_1, x_2) \in \Gamma}\}$ span only $3$ dimensional space, where $\Gamma=\{(x_1, x_2): P(x_1, x_2)=1\}$, and where $P$ is a (homogeneous) polynomial of degree $5$?  If so, then there exists matrix
$D=\{d_{ij}\}_{i=1,\dots, N-3; j=1, \dots, N}$, whose rank is $N-3$, such that
$$
De_4(x_1, x_2) =0, \quad \forall (x_1, x_2): P(x_1, x_2)=1\,.
$$
This is possible only if 
\begin{equation}
\label{equiv5}
De_4(x_1, x_2) \equiv 0\quad  \forall (x_1, x_2).
\end{equation}

Explanation: take a row of $D$, $d=(d_1, d_2, \dots, d_N)$, let $Q(x_1, x_2) :=
d_1x_1^4+ d_2 x_2^4+ d_3Q_3(x_1, x_2)+\dots+ d_n Q_N(x_1, x_2)$.  From the displays above we get that as soon as $P(x_1, x_2)=1$ we have $Q(x_1, x_2)=0$, where $Q$ is a homogeneous polynomial of degree $4$. This is of course impossible, unless $Q\equiv 0$. In fact, genuine curve of degree $5$ $P=1$ on the plane cannot sit inside the zero set of homogeneous polynomial of degree $4$, when $P$ is a homogeneous polynomial of degree $5$.

\bigskip

So we start with \eqref{equiv5}. Assume now that the invertible square $(N-3)\times (N-3)$ minor $\hat  D$ of $D$ is the first $N-3$ columns of $D$ (it is really WLOG, as the reader will see from what follows).

Coming back to \eqref{equiv}: $De_4(x_1, x_2) \equiv 0\,\, \forall (x_1, x_2)$, we can now make a linear change of variable:  $\gamma_{*, i} \to \tilde\gamma_{*, i}$, $i=1,2$, in such a way that
$$
\tilde\gamma_{N-1, 1} =1, \tilde\gamma_{N, 1} =0; \quad \tilde\gamma_{N-1, 2} =0, \tilde\gamma_{N, 2} =1\,.
$$
Then we can rewrite  $De_4(x_1, x_2) \equiv 0\,\, \forall (x_1, x_2)$ as equation with forms $\tilde Q_k:=(\tilde\gamma_{k1} x_1+\tilde\gamma_{k2} x_2)^4$,$k=1,\dots, N-2$, $\vec Q:= \{\tilde Q_k\}_{k=1,\dots, N-2}$, $\vec d_{N-1} := \{d_{k, N-1}\}_{k=1,\dots, N-2}$,  $\vec d_{N} := \{d_{k, N}\}_{k=1,\dots, N-2}$,
$\tilde D= (\hat D, \vec d_{N-2})$ as follows

\begin{eqnarray}
\label{532}
&\hat D \vec Q =  \vec d_{N-1} (\tilde\gamma_{N-1, 1} x_1 + \tilde\gamma_{N-1, 2}x_2)^4 +  \vec d_{N} (\tilde\gamma_{N, 1} x_1 + \tilde\gamma_{N, 2}x_2)^4=\notag
\\
&= x_1^4 \vec d_{N-1} + x_2^4 \vec d_{N}
\end{eqnarray}
 The right hand side does not have terms $x_1^3x_2, x_1^2x_2^2, x_1x_2^3$. But matrix $\tilde D$ has size $(N-3)\times (N-2)$. And its $N-3$ rows span a space of co-dimension $1$ in $\bR^{N-2}$. Hence all orthogonal vectors to this span are co-linear. There are three such vectors:
 $$
 \vec v_{22}:= (\tilde \gamma_{11}^2\cdot \tilde\gamma_{12}^2, \dots, \tilde\gamma_{N-2,1}^2\cdot \tilde\gamma_{N-2,2}^2),
 $$
 $$
 \vec v_{31}:= \tilde(\gamma_{11}^3\cdot \tilde\gamma_{12}, \dots, \tilde\gamma_{N-2,1}^3\cdot \tilde\gamma_{N-2,2}),
 $$
$$
 \vec v_{13}:= (\tilde\gamma_{11}\cdot \tilde\gamma_{12}^3, \dots, \tilde\gamma_{N-2,1}\cdot\tilde \gamma_{N-2,2}^3).
 $$
 Obviously they are co-linear only if vector $(\tilde \gamma_{11},\dots, \tilde \gamma_{N-2, 1})$ is proportional to
 \linebreak $(\tilde \gamma_{12},\dots, \tilde \gamma_{N-2, 2})$ or if some of these $\tilde\gamma$'s are zeros.
 We get  algebraic relationships (akin to \eqref{0} of the previous Section)  on  entries of matrix $C$ that gave us  bad subspace $L$, as in the previous Section.
  So as soon as $2$-dimensional subspace $L$ (parametrized by $C$ above) has $D$ as above (i.e. the space spanned by normals to $N(S, L)$ is also $3$-dimensional) immediately we have certain algebraic relationship on entries of matrix $C$.  That means that the dimension of set of bad $L$'s is strictly smaller than dimension of all $L$'s. This gives us the proof of Lemma \ref{ell54}.
 \end{proof}

Now we repeat verbatim Subsection  \ref{cex}. Then we get
\begin{theorem}
\label{thm32}
In $\ell^5_N$ there exists a subspace of co-dimension $3$ such that there is no $2$ dimensional subspace that is orthogonal to it.
\end{theorem}
 
\medskip

\subsection{Bookkeeping}

We can use $\ell^5_N$ to build a $(4, 2)$ counterexample and  for $(4,3)$ counterexample. For $(5, 2)$  counterexamples we will  need $\ell^7_N$.
Et cetera. 

In  fact, if we wish to construct a $(m, k)$ counterexample for $1<k\le m$, we take large $N$, say, $N= 2m$, and consider the first odd number $2s+1$ such that
$$
2s+1 >m-k+2.
$$
Then we can construct  a subspace $E\subset \ell^{2s+1}_N$, of co-dimension $m$, such that there is no subspaces of $\ell^{2s+1}_N$ of dimensions $k$ orthogonal to $E$.

If we want one $E$ that serves all $k=2, 3, \dots, m$ simultaneously, we choose $2s+1>m$, and in $ \ell^{2s+1}_N$ one can construct  such a subspace $E$.

\medskip

Now consider the space
$$
X=\bigoplus_{m=1}^\infty \ell^{2m+1}_{2m}\,.
$$
In this space we can find the sequence $F_m, E_m$ of finitely dimensional subspaces, such that $\dim F_m-\dim E_m =m$, and we can find a sequences of numbers $\{\eps_m\}$, $\eps_m\to 0$, such that  for any subspace  $K_m$ of $F_m$, the orthogonality of $K_m$ to $E_m$ is {\it  worse} than $1-\eps_m$.

\medskip

{\bf Open question.}
So, we cannot find {\it very orthogonal} $K_m$.  But maybe this sequence  $\{\eps_m\}$ above goes to zero too fast? So, still maybe, given a space $X$ and given $F_m, E_m$ of finitely dimensional subspaces, such that $\dim F_m-\dim E_m =m$,  we can always  find {\it some} slowly going to zero sequence $\{\eps_m\}$, such that 
$K_m$ that is $1-\eps_m$ orthogonal to $E_m$ does exist. This we do not know.

\section{Why that was counterintuitive, Grassmannians and their cohomology}
\label{cohom}

The result of the previous Section seems counterintuitive. Let us explain why it seems counterintuitive.

There are elementary proofs of Borsuk--Ulam theorem, but most common proofs go through the following reasoning.
A simple commutative diagram shows that if there is an odd  continuous  map $S^N$ into $S^n$, $N>n$, then
there is \cite{F} 1) a continuous map of real  projective space $\bR\bP^N$ into $\bR\bP^n$ such that 2) the induced map between cohomology rings $f^*:H^*(\bR\bP^n, \bZ_2) \to H^*(\bR\bP^N, \bZ_2)$ behaves like that: 
\begin{equation}
\label{prop}f^*(x)=y\,.
\end{equation}
Having in mind that
$$
H^*(\bR\bP^k) = \bZ_2[x]/(x^{k+1})
$$
we get  $0= f^*(x^{n+1}) = y^{n+1} \neq 0$,
bringing the contradiction.

Real projective space $\bR\bP^n$ is just a Grassmannian $G(1,\bR^n)$. To build a one-dimensional space in $F$ that is ``orthogonal'' to the whole $n$-dimensional $E$ (see Theorem \ref{borsuk}) one builds an odd continuous map $S^N$ into $S^n$, $N>n$, $N=\dim F$, $n=\dim E$. In other words, one builds a special map from $G(1,\bR^N)$ to $G(1,\bR^n)$, whose induced cohomology ring structure mapping has extra property \eqref{prop}.

\medskip

However, we need more, namely we need to build not just a $1$-dimensional space of $F$ ``orthogonal'' to $E$ but a {\it high} dimensional space of $F$ ``orthogonal'' to $E$. Hopefully  even $\dim F-\dim E= N-n$-dimensional one.

For that we need to see how, assuming that there is no such $k$-dimensional  subspace in $F$, $k>1$, ``orthogonal'' to $E$, one can construct the continuous map of Grassmannians $G(k,\bR^N)$  to $G(k,\bR^n)$ with the property of induced map $f^*$ on total cohomology  rings (which are known and computed) that brings the contradiction as above.

\medskip

The main thing is to construct a geometrically meaningful continuous map of Grassmannians $G(k,\bR^N)$  to $G(k,\bR^n)$ given that  there is no such $k$-dimensional  subspace in $F$, $k>1$, ``orthogonal'' to $E$. For $k=1$ this has been done--easily--in the proofs of Theorem \ref{borsuk}.
But  we constructed exactly the situation when there is no such $k$-dimensional  subspace in $F$, $k>1$, ``orthogonal'' to $E$ even just for $k=2$. So the hope to use the algebraic topology approach to prove some generalization of Theorem \ref{kkm}  seems to be vain.

\section{Background - Fredholm Theory and Orthogonal subspaces}
\label{fredholm}

While questions 1 and 2 considered in the previous sections relate to the fundamental properties of finite-dimensional Banach spaces, these questions came up as a result of investigation of stability of Fredholm operators and pairs of subspaces. The following briefly explains this context.

Traditionally, stability of operators has been studied by perturbing a Fredholm operator by a small norm or compact operator, or by an operator with a small measure of noncompactness. It has also been established for strictly singular perturbations. Fredholm properties of a pair of subspaces in Banach spaces have been studied in the gap topology and for strictly singular subspace perturbations. In another trend, stability of operators was studied for collectively compact operators or, equivalently, compact approximation. 

The work \cite{BB} contains a list of  references to the original contributions in these areas from the first half of the 20th century until recent times. In that work, a new way of looking at the Fredholm stability in Banach spaces has been proposed, which subsumed many of the existing at that time directions. Several Fredholm stability theorems have been shown to be still true in this general setting. However, a few questions had arisen that have not been resolved so far.

To explain the unresolved questions, we only need two definitions from \cite{BB} (note that below $\mathbb{N}^{'''} \subset\ \mathbb{N}^{''} \subset\ \mathbb{N}^{'}$ are always infinite subsets of the set of natural numbers $\mathbb{N}$).

\begin{definition}[$\lambda-$Adjustment of Sequences of Subspaces]\label{D:ulass} 
Let $(M_{n})_{\mathbb{N}^{'}}$ and $(P_{n})_{\mathbb{N}^{'}}$ be a pair of sequences of closed subspaces from a Banach space $X$, ${M_{n} \neq \{\theta\}}$ for all ${n \in \mathbb{N}^{'}}$ and $\lambda \geq 0$. We say that $(M_{n})_{\mathbb{N}^{'}}$ is $\lambda-adjusted$ with $(P_{n})_{\mathbb{N}^{'}}$ if for any $\eta > 0$ and for any unit subsequence $(x_{n})_{\mathbb{N}^{''}}$ from $(M_{n})_{\mathbb{N}^{''}}$ there exists a subsequence $(y_{n})_{\mathbb{N}^{'''}}$ from $(P_{n})_{\mathbb{N}^{'''}}$ and a vector $z \in X$ such that 
\[
\varlimsup_{n \in \mathbb{N}^{'''}} \left\|x_{n} - y_{n} - z \right\|\ \leq\ \lambda + \eta. 
\]
The $\lambda-$adjustment between $(M_{n})_{\mathbb{N}^{'}}$ and $(P_{n})_{\mathbb{N}^{'}}$ is a non-negative real number defined as
\[
\lambda_{\mathbb{N}^{'}}[M_{n}, P_{n}]\ :=\ \inf \{ \lambda \in \mathbb{R} \mid (M_{n})_{\mathbb{N}^{'}}\ \text{is $\lambda-$adjusted with}\ (P_{n})_{\mathbb{N}^{'}}\}.
\]
\end{definition}

Observe that above $z$ can be any vector from $X$. If $z$ is always a $null$ vector, then the definition becomes similar to the proximity in the sense of the gap distance as defined in \cite{KKM}. However, when $z$ is not always $null$, the two adjusted sequences of subspaces can be far apart in the gap sense. Still, certain Fredholm stability properties can be established for both the pairs of subspaces and for operators when considering sequences of operator graphs as sequences of subspaces in the product space of their domain and range.

An extension of the $\lambda-$adjustment can be achieved via the following Definition that is less restrictive than the previous one (note that below $(K_{n})_{\mathbb{N}^{''}} \prec (M_{n})_{\mathbb{N}^{''}}$ means each $K_{n}$ is a subspace of $M_{n}$ for all $n$) 

\begin{definition}[Finitely Strictly Singular $\lambda-$Adjustment]\label{ssua}
Let $(M_{n})_{\mathbb{N}^{'}}$ and $(P_{n})_{\mathbb{N}^{'}}$ be a pair of sequences of closed subspaces from a Banach space $X$, ${M_{n} \neq \{\theta\}}$ for all ${n \in \mathbb{N}^{'}}$ and $\lambda \geq 0$. We say that $(M_{n})_{\mathbb{N}^{'}}$ is $finitely$\ $strictly$\ $singular$\ $\lambda-adjusted$ with $(P_{n})_{\mathbb{N}^{'}}$ if for any subsequence of subspaces $(K_{n})_{\mathbb{N}^{''}} \prec (M_{n})_{\mathbb{N}^{''}}$ such that ${\dim K_{n} < \infty}$ for all ${n \in \mathbb{N''}}$ and ${\lim_{n \in \mathbb{N}^{''}} \dim K_{n} = \infty}$ there exists a subsequence of subspaces $(L_{n})_{\mathbb{N}^{'''}} \prec (K_{n})_{\mathbb{N}^{'''}}$ such that ${\lim_{n \in \mathbb{N}^{'''}} \dim L_{n} = \infty}$ with the property ${\lambda_{\mathbb{N}^{'''}}[L_{n}, P_{n}] \leq \lambda}$. 
\end{definition}

The above definition subsumes finitely strictly singular operators and pairs of subspaces. Again, a semi-Fredholm stability can be proven for $\lambda = 0$ or small enough positive $\lambda$. An interesting example of such a sequence are the graphs of compositions of a finitely strictly singular operator with a sequence of the linear operators having a common norm bound - such compositions are finitely strictly singular $0-$adjusted with a $null$ operator $\theta$.

It is instrumental to consider both kinds of $\lambda-$adjustment when the subspaces in question are finite-dimensional. For example, one can prove that if dimensions of all $P_{n}$ are limited from above, then there exists no $(M_{n})_{\mathbb{N}^{'}}$ that is finitely strictly singular $\lambda$-adjusted with $(P_{n})_{\mathbb{N}^{'}}$ if $\lambda < 1/2$. Yet, when dimensions of $(P_{n})_{\mathbb{N}^{'}}$ approach infinity, the question arises as to a relation between these dimensions and dimensions of $(M_{n})_{\mathbb{N}^{'}}$. The behavior in case of a $\lambda-$adjustment appears to differ from finitely strictly singular $\lambda-$adjustment.

One can prove the following:

\begin{lemma}
\label{limit}
For $\lambda < 1/2$ and $(M_{n})_{\mathbb{N}^{'}}$ which is $\lambda-$adjusted with finite-dimensional $(P_{n})_{\mathbb{N}^{'}}$, there exists a constant integer $C$ such that $\dim M_{n} < \dim P_{n} + C$ for all large enough $n$.
\end{lemma}

Note that the above Lemma \ref{limit} is not true in $X = \ell_{\infty}$ when $\lambda \geq 1/2$. Still, with some additional effort, it is possible to prove this Lemma for any $\lambda < 1$ in a Banach space $X$ which dual is a Fr\'{e}chet-Urysohn topological space in ${weak^{*}}$ topology. An interesting question remains if this limiting dimension property for $\lambda < 1$ characterizes such Banach spaces?

However, for finitely strictly singular $\lambda-$adjustment in Banach spaces, Lemma \ref{limit} remains unproven (although it is obviously true for Hilbert spaces). In other words, it is not clear if the following is true:

\begin{conj}
\label{unknown}
For $\lambda < 1$ and $(M_{n})_{\mathbb{N}^{'}}$ which is finitely strictly singular $\lambda-$adjusted with finite-dimensional $(P_{n})_{\mathbb{N}^{'}}$, there exists a constant integer $C$ such that $\dim M_{n} < \dim P_{n} + C$ for all large enough $n$.
\end{conj}

The answers to questions $1$ and $2$ from the previous sections put a stop on an attempt to prove Conjecture \ref{unknown} where one would look for the ``orthogonal'' to $P_{n}$ subspaces of increasingly high dimensions inside $M_{n}$.

Also note that the counter-examples from the previous sections are of an independent interest as they show that the $1$-dimensional orthogonality Theorem 2 from \cite{KKM} (see Theorem \ref{kkm}) cannot be extended to higher dimensions.


\begin{thebibliography}{999}

\bibitem[Bohn]{Bohn} {\sc F. Bohnenblust,} {\em Convex regions and projections in Minkowski spaces}, Ann. of Math. 39 (1938), 301--308.

\bibitem[B]{B}  {\sc B. Bollob\'as}, {\em The art of mathematics: Coffee time in Memphis}, New York: Cambridge University Press, 2--6.

\bibitem[Bo]{Bo} {\sc  K. Borsuk,} {\em Drei S\"atze \"uber die $n$-dimensionale euklidische Sph\"are},  Fundamenta Mathematicae (in German) 20: 177--190. doi:10.4064/fm-20-1-177-190.

\bibitem[BB]{BB} {\sc B. I. Burshteyn}, {\em Strictly Singular Uniform $\lambda-$Adjustment in Banach Spaces}. Arxiv:math.FA/0902.3045 Vol. {\bf 1}, (18 Feb 2009).

\bibitem[Fr]{Fr}{\sc C. Franchetti} {\em  The norm of the minimum projection onto hyperplanes in $L^p[0,1]$ and the
radial constant}, Boll. Un. Mat. Ital. 4-B(7) (1990), 803--821.

\bibitem[F]{F} {\sc M. L. Fries}, {\em Borsuk--Ulam theorem and applications presented by Alex Suciu and Marcus Fries}.  Preprint.

\bibitem[JL]{JL} {\sc W. B. Johnson,  J. Lindenstrauss}, {\em Handbook of the Geometry of Banach Spaces}. North-Holland 2001 Vol. {\bf 1-2}.

\bibitem[KT]{KT} {\sc H. K\"oonig and N. Tomczak-Jaegermann,} {\em Norms of minimal projections}, J. Funct. Anal. 119 (1994), 253--280.

\bibitem[KST]{KST} {\sc H. K\"oonig , C. Schuett and N. Tomczak-Jaegermann}, {\em Projection constants of symmetric spaces and variants of Khintchines inequality}, J. Reine Angew. Math. 511 (1999), 1--42.

\bibitem[LO]{LO} {\sc G. Lewicki and W. Odyniec}, {\em Minimal projections in Banach Spaces}, Lectures Notes in Mathematics, 1449 (Springer-Verlag, Berlin, 1991).


\bibitem[TK]{TK}{\sc T. Kobos}, {\em Hyperplane of finite-dimensional normed spaces with the maximal relative projection constant}, arXiv:1411.6214, pp. 1--15.

\bibitem[KKM]{KKM} {\sc M. Krein, M. Krasnoselski, D. Milman}, {\em On defect numbers of operators on Banach spaces and related geometric problems}, [in Russian], Trudy Inst. Mat. Akad. Nauk Ukrain. SSR, 11 (1948), 97--112.

\bibitem[LS]{LS} {\sc  L. Lusternik,  L. Schnirelmann}, {\em M\'ethodes topologiques dans les probl\`emes variationnels}, Moscow: Gosizdat, Moscow,1930.

\bibitem[R]{R}Joseph J. Rotman, An Introduction to Algebraic Topology (1988) Springer-Verlag ISBN 0-387-96678-1.







\end{thebibliography}
\end{document}